\documentclass{amsart}
\usepackage{amssymb}
\usepackage{hyperref}
\usepackage{ifthen}

\newtheorem{theorem}{Theorem}[section]

\newtheorem{lemma}[theorem]{Lemma}
\newtheorem{remark}[theorem]{Remark}
\newtheorem{example}[theorem]{Example}
\theoremstyle{definition}
\newtheorem{definition}[theorem]{Definition}
\numberwithin{equation}{section} 
\newtheorem{proposition}[theorem]{Proposition}

 \DeclareMathOperator\Aut{{\rm Aut}}

 \DeclareMathOperator\Inn{{\rm Inn}}
 
 \DeclareMathOperator\Out{{\rm Out}}

 \DeclareMathOperator\rk{{\rm rk}}

\def\Sp{{\rm Sp}}

\newcommand\C{\mathbb{C}}

\newcommand\Gl{{\rm G \ell}}

\newcommand\LA{{\rm A}}

\newcommand\LE{{\rm E}}
\newcommand\LF{{\rm F}}
\newcommand\LG{{\rm G}}
\newcommand\La{\mathfrak a}

\newcommand\Le{\mathfrak e}
\newcommand\Lf{\mathfrak f}
\newcommand\Lg{\mathfrak g}

\newcommand\R{\mathbb{R}}
\newcommand\SO{{\rm SO}}
\newcommand\SU{{\rm SU}}
\newcommand\Sl{{\rm S \ell}}
\newcommand\Spin{{\rm Spin}}

\newcommand\U{{\rm U}}

\newcommand\Z{\mathbb{Z}}
\newcommand\ZZ{\Z_2 \times \Z_2}
\newcommand\ZZs{$\ZZ$--symmetric space}
\newcommand\ZZss{$\ZZ$--symmetric spaces}
\newcommand\act{{\,\mbox{\raisebox{0.2em}{\bf .}}\,}}
\newcommand\bl{\vspace{1em}\hfill\\}

\newcommand\diag{{\rm diag}}

\newcommand\e{{\rm e}}
\newcommand\gl{\g\ell}
\newcommand\gc{\g_{\C}}
\newcommand\g{{\mathfrak g}}
\newcommand\gs{{\mathfrak g}^{\s}}
\newcommand\gst{{\mathfrak g}^{\s\t}}
\newcommand\gt{{\mathfrak g}^{\t}}

\newcommand\h{{\mathfrak h}}
\newcommand\id{{\rm id}}

\newcommand\pp{{\mathfrak p}}

\newcommand\s{\sigma}
\newcommand\so{\mathfrak s \mathfrak o}
\newcommand\spin{\mathfrak s \mathfrak p \mathfrak i \mathfrak n}

\newcommand\str{\rule[-.48em]{0em}{1.7em}}
\newcommand\suxu[2]{\mathfrak s(\mathfrak u({#1}) {+} \mathfrak u({#2}))}

\newcommand\su{\mathfrak s \mathfrak u}

\renewcommand\O{{\rm O}}

\renewcommand\a{\alpha}

\renewcommand\k{{\mathfrak k}}

\renewcommand\sl{\mathfrak s \mathfrak l}
\renewcommand\sp{\mathfrak s \mathfrak p}

\renewcommand\t{\tau}
\renewcommand\tt{\mathfrak t}
\renewcommand\u{\mathfrak u}

\newcommand{\lp}{\left (}                     % left  parenthesis
\newcommand{\rp}{\right )}                    % right parenthesis
\newcommand{\lbk}{\left [}                    % left  square bracket
\newcommand{\rbk}{\right ]}                   % right square bracket
\newcommand{\lbc}{\left \{}                   % left  braces
\newcommand{\rbc}{\right \}}                  % right braces
\renewcommand\SU[1]{{\rm SU}{\ifthenelse{\equal{#1}{(}}{(}{\left(#1\right)}}}
\renewcommand\Sp[1]{{\rm Sp}{\ifthenelse{\equal{#1}{(}}{(}{\left(#1\right)}}}
\renewcommand\U[1]{{\rm U}{\ifthenelse{\equal{#1}{(}}{(}{\left(#1\right)}}}
\renewcommand\SO[1]{{\rm SO}{\ifthenelse{\equal{#1}{(}}{(}{\left(#1\right)}}}
\renewcommand\O[1]{{\rm O}{\ifthenelse{\equal{#1}{(}}{(}{\left(#1\right)}}}
\renewcommand\Spin[1]{{\rm Spin}{\ifthenelse{\equal{#1}{(}}{(}{\left(#1\right)}}}
\renewcommand\Gl[1]{{\rm G\ell}{\ifthenelse{\equal{#1}{(}}{(}{\left(#1\right)}}}
\renewcommand\Sl[1]{{\rm S\ell}{\ifthenelse{\equal{#1}{(}}{(}{\left(#1\right)}}}
\renewcommand\su[1]{{\mathfrak s \mathfrak u}
 {\ifthenelse{\equal{#1}{(}}{(}{\left(#1\right)}}}
\renewcommand\sp[1]{{\mathfrak s \mathfrak p}
 {\ifthenelse{\equal{#1}{(}}{(}{\left(#1\right)}}}
\renewcommand\u[1]{{\mathfrak u}
 {\ifthenelse{\equal{#1}{(}}{(}{\left(#1\right)}}}
\renewcommand\so[1]{{\mathfrak s \mathfrak o}{\ifthenelse{\equal{#1}{(}}{(}{\left(#1\right)}}}
\renewcommand\o[1]{{\mathfrak o}{\ifthenelse{\equal{#1}{(}}{(}{\left(#1\right)}}}
\renewcommand\spin[1]{{\mathfrak s \mathfrak p \mathfrak i \mathfrak n}
 {\ifthenelse{\equal{#1}{(}}{(}{\left(#1\right)}}}
\renewcommand\gl[1]{{\mathfrak g \mathfrak l}{\ifthenelse{\equal{#1}{(}}{(}{\left(#1\right)}}}
\renewcommand\sl[1]{{\mathfrak s \mathfrak l}{\ifthenelse{\equal{#1}{(}}{(}{\left(#1\right)}}}

\begin{document}
\setcounter{tocdepth}{1}

%%%%%%%%%%%%%%%%%%%%%%%%%%%%%%%%%%%%%%%%%%%%%%%%%%%%%%%%%%%%%%%%
%       A user defined counter for numbering
%       and labeling items in tables etc.
%%%%%%%%%%%%%%%%%%%%%%%%%%%%%%%%%%%%%%%%%%%%%%%%%%%%%%%%%%%%%%%%

\newcounter{refcounter}
\renewcommand\therefcounter{\bf (\arabic{refcounter})}
\newcommand\mylabel[1]{\refstepcounter{refcounter}\bf\therefcounter\label{#1}}

%%%%%%%%%%%%%%%%%%%%%%%%%%%%%%%%%%%%%%%%%%%%%%%%%%%%%%%%%%%
%%%%%%%%%%%%%%%%%%%%%%%%%%%%%%%%%%%%%%%%%%%%%%%%%%%%%%%%%%%
%
%       T I T L E,    A B S T R A C T   &
%
%       A U T H O R    I N F O R M A T I O N
%
%%%%%%%%%%%%%%%%%%%%%%%%%%%%%%%%%%%%%%%%%%%%%%%%%%%%%%%%%%%
%%%%%%%%%%%%%%%%%%%%%%%%%%%%%%%%%%%%%%%%%%%%%%%%%%%%%%%%%%%

\pagenumbering{arabic}\pagestyle{plain}

\title{Exceptional $\Z_2 \times \Z_2$--symmetric spaces}

\author{Andreas Kollross}

\address{Institut f\"{u}r Mathematik\\Universit\"{a}t Augsburg\\86135
Augsburg\\Germany}

\email{kollross@math.uni-augsburg.de}

\subjclass[2000]{53C30, 53C35; 17B40}

\keywords{exceptional $\Z_2 \times \Z_2$--symmetric space, Lie algebra grading}

\begin{abstract}
The notion of $\Z_2 \times \Z_2$--symmetric spaces is a generalization of
classical symmetric spaces, where the group $\Z_2$ is replaced by $\Z_2 \times \Z_2$. 
In this article, a classification is given of the
$\Z_2 \times \Z_2$--symmetric spaces $G/K$ where $G$ is an exceptional compact
Lie group or $\Spin(8)$, complementing recent results of Bahturin and Goze.
Our results are equivalent to a classification of $\Z_2 \times \Z_2$--gradings
on the exceptional simple Lie algebras $\Le_6, \Le_7, \Le_8, \Lf_4, \Lg_2$ and $\so8$.
\end{abstract}

\maketitle

%%%%%%%%%%%%%%%%%%%%%%%%%%%%%%%%%%%%%%%%%%%%%%%%%%%%%%%%%%%%%%%%%%%%%%%%%%%%%%%%

%%%\tableofcontents

\section{Introduction and results}
\label{Intro}

%%%%%%%%%%%%%%%%%%%%%%%%%%%%%%%%%%%%%%%%%%%%%%%%%%%%%%%%%%%%%%%%%%%%%%%%%%%%%%%%

The notion of $\Gamma$--symmetric spaces introduced by Lutz~\cite{lutz} is a
generalization of the classical notion of a symmetric space.

%%%%%%%%%%%%%%%%%%%%%%%%%%%%%%%%%%%%%%%%%%%%%%%%%%%%%%%%%%%%%%%%%%%%%%%%%%%%%%%%

\begin{definition}
Let $\Gamma$ be a finite abelian group and let $G$ be a connected Lie group. A
homogeneous space~$G/K$ is called {\em $\Gamma$--symmetric} if $G$ acts almost
effectively on~$G/K$ and there is an injective homomorphism $\rho \colon \Gamma
\to \Aut G$, such that $G^{\Gamma}_0 \subseteq K \subseteq G^{\Gamma}$, where
$G^{\Gamma}$ is the subgroup of elements fixed by~$\rho(\Gamma)$ and
$G^{\Gamma}_0$ its connected component.
\end{definition}

%%%%%%%%%%%%%%%%%%%%%%%%%%%%%%%%%%%%%%%%%%%%%%%%%%%%%%%%%%%%%%%%%%%%%%%%%%%%%%%%

In the case $\Gamma = \Z_2$ this is just the classical definition of symmetric
spaces, in case $\Gamma = \Z_k$ one obtains $k$--symmetric spaces, as studied
in~\cite{graywolf}. In case $\Gamma = \ZZ$ we can rephrase the definition as
follows. A homogeneous space~$G/K$ is {\em $\ZZ$--symmetric} if and only if
there are two different commuting involutions on~$G$, i.e.\ there are $\sigma,
\tau \in \Aut G \setminus \lbc \id_G \rbc$ such that $\sigma^2 = \tau^2 =
\id_G$, $\sigma \neq \tau$ and $\sigma \tau = \tau \sigma$ such that $\lp
G^{\sigma} \cap G^{\tau} \rp_0 \subseteq K \subseteq G^{\sigma} \cap G^{\tau}$.

It is the purpose of this paper to give a classification of \ZZss\ in case $G$
is a simply connected compact Lie group of isomorphism type $\Spin8$, $\LE_6$,
$\LE_7$, $\LE_8$, $\LF_4$ or $\LG_2$. This amounts to a classification of pairs
of commuting involutions on the Lie algebras of these groups. Recently, such a
classification has been obtained by Bahturin and Goze~\cite{bg} for the case of
simple classical Lie algebras except~$\so8$.  The results of the classification
in the exceptional and $\Spin8$ case are given in Theorems~\ref{ThMain1} and
\ref{ThMain2}, respectively, below.

Let $G/K$ be a \ZZs. Then there is a triple of involutions ${\s}$, ${\t}$,
${\s\t}$ of~$G$ and we say that the \ZZs\ is of {\em type} $\lp X,Y,Z \rp$
where $X$, $Y$, $Z$ denote the local isomorphism classes of the symmetric
spaces $G/G^{\s}$, $G/G^{\t}$ and $G/G^{\s\t}$, respectively. For example, we
will show that there are two commuting involutions $\s,\t$ on $G = \LE_7$ such
that $G/G^{\s}$, $G/G^{\t}$ and $G/G^{\s\t}$ are isomorphic to symmetric spaces
of type E\,V, E\,VI and E\,VII, respectively and such
that the Lie algebra of $\LE_6^{\s} \cap \LE_6^{\t}$ is isomorphic to
$\su6 + \sp1 + \R$, cf.\ Table~\ref{TInv}. Thus we say that the corresponding \ZZs\
is of type $\lp \mbox{E\,V, E\,VI, E\,VII}\rp$, abbreviated as E\,V-VI-VII.

%%%%%%%%%%%%%%%%%%%%%%%%%%%%%%%%%%%%%%%%%%%%%%%%%%%%%%%%%%%%%%%%%%%%%%%%%%%%%%%%
%%%%%%%%%%%%%%%%%%%%%%%%%%%%%%%%%%%%%%%%%%%%%%%%%%%%%%%%%%%%%%%%%%%%%%%%%%%%%%%%
\begin{table}[!h]\rm
\begin{tabular}{|l|c|c|}
\hline \str Type & $\g$ & $\k$  \\
\hline \hline

 E\,I-I-II & $\Le_6$ &  $\so6 + \R$ \\ \hline

 E\,I-I-III & $\Le_6$ &  $\sp2 + \sp2$ \\ \hline

 E\,I-II-IV & $\Le_6$ &  $\sp3 + \sp1$ \\ \hline

 E\,II-II-II & $\Le_6$ &  $\su3 + \su3 + \R + \R$ \\ \hline

 E\,II-II-III & $\Le_6$ &  $\su4 + \sp1 + \sp1 + \R$ \\ \hline

 E\,II-III-III & $\Le_6$ &  $\su5 + \R + \R$ \\ \hline

 E\,III-III-III & $\Le_6$ &  $\so8 + \R + \R$ \\ \hline

 E\,III-IV-IV & $\Le_6$ &  $\so9$ \\ \hline \hline
%

%-----------------------------------------------------------

 E\,V-V-V & $\Le_7$ &  $\so8$ \\ \hline

 E\,V-V-VI & $\Le_7$ &  $\su4 + \su4 + \R$ \\ \hline

 E\,V-V-VII & $\Le_7$ &  $\sp4$ \\ \hline

 E\,V-VI-VII & $\Le_7$ & $\su6 + \sp1 + \R$ \\ \hline

 E\,VI-VI-VI &  $\Le_7$ & $\so8 + \so4 + \sp1$ \\

             &  $\Le_7$ & $\u6 + \R$ \\ \hline

 E\,VI-VII-VII &  $\Le_7$ & $\so(10) + \R + \R$ \\ \hline

 E\,VII-VII-VII &  $\Le_7$ & $\Lf_4$ \\ \hline \hline
%

%-----------------------------------------------------------

 E\,VIII-VIII-VIII & $\Le_8$ &  $\so(8) + \so8$ \\ \hline

 E\,VIII-VIII-IX & $\Le_8$ &  $\su8 + \R$ \\ \hline

 E\,VIII-IX-IX & $\Le_8$ &  $\so(12) + \sp1 + \sp1$ \\ \hline

 E\,IX-IX-IX & $\Le_8$ &  $\Le_6 + \R + \R$ \\ \hline \hline
%

%-----------------------------------------------------------

 F\,I-I-I & $\Lf_4$ &  $\u3 + \R$ \\ \hline

 F\,I-I-II & $\Lf_4$ &  $\sp2 + \sp1 + \sp1$ \\ \hline

 F\,II-II-II & $\Lf_4$ &  $\so8$ \\ \hline \hline
%

%-----------------------------------------------------------

 G & $\Lg_2$ & $\R + \R$ \\ \hline
\end{tabular}
\bl

\caption{Exceptional $\ZZ$--symmetric spaces.} \label{TExcSpaces}

\end{table}
%%%%%%%%%%%%%%%%%%%%%%%%%%%%%%%%%%%%%%%%%%%%%%%%%%%%%%%%%%%%%%%%%%%%%%%%%%%%%%%%

The problem of commuting involutions has a geometric interpretation in
terms of classical $\Z_2$--symmetric spaces.  Let $G$ be a simple compact
connected Lie group and let $H$, $L$ be symmetric subgroups of~$G$ (see
Definition~\ref{DefSymSub} below) such that $\lp G^{\s} \rp_0 \subseteq H
\subseteq G^{\s}$ and $\lp G^{\t} \rp_0 \subseteq L \subseteq G^{\t}$, where
$\s,\t \in \Aut(G)$ are involutions. Then $M = G/L$ endowed with a left
$G$-invariant metric is a symmetric space in the classical sense and there is a
natural isometric action of $H$ on $M$ which is hyperpolar, i.e.\ there exists
a flat submanifold which meets all the orbits and such that its intersections
with the orbits are everywhere orthogonal, see \cite{hptt} and
\cite{hyperpolar}. These actions have first been considered by
Hermann~\cite{hermann1}; they are called {\em Hermann actions}.

In \cite{hermann2}, Hermann proved that the $H$--orbit through $g \, L \in M$,
$g \in G$ is totally geodesic if and only if the involutions $i_g\, \s\,
i_g^{-1}$ and $\t$ commute where we use the notation $i_g$ to denote the
automorphism of a Lie group~$G$ given by $i_g(x) = g\,x\,g^{-1}$; this means
that the classification of commuting involutions on~$G$ is equivalent to a
classification of totally geodesic orbits of Hermann actions.

Conlon~\cite{conlon} determined all pairs of conjugacy classes of involutions
$\s,\t$ for which there is a $g \in G$ such that $i_g \s i_g^{-1}$ and $\t$
commute. The proof relies on the unpublished notes~\cite{conlon2}.
(For inner involutions $\s = i_a$, $\t = i_b$ there is obviously always
such a $g$ since there is a $g \in G$ such that $g\,a\,g^{-1}$ and $b$ are
contained in one and the same maximal torus of~$G$.) Our classification however
uses a direct approach which does not rely on any of the previously
mentioned results.

\begin{theorem}\label{ThMain1}
Let $\g$ be a compact exceptional Lie algebra of type $\Le_6$, $\Le_7$,
$\Le_8$, $\Lf_4$, or $\Lg_2$. If $\s,\t \in \Aut(\g)$ are two commuting
involutions such that $\s \neq \t$ then the pair $(\g,\gs \cap \gt)$ is one of
the pairs $(\g,\k)$ given by Table~\ref{TExcSpaces}; the conjugacy classes of
$\s$, $\t$ and $\s\t$ are given by the first row of Table~\ref{TExcSpaces}.
Conversely, for any pair of Lie algebras $(\g,\k)$ in Table~\ref{TExcSpaces}
there exists a pair $(\s,\t)$ of commuting involutions of~$\g$ such that $\k
\cong \gs \cap \gt$.
\end{theorem}

%%%%%%%%%%%%%%%%%%%%%%%%%%%%%%%%%%%%%%%%%%%%%%%%%%%%%%%%%%%%%%%%%%%%%%%%%%%%%%%%

The data in Table~\ref{TExcSpaces} in all cases but one determines the
conjugacy class of the subalgebra $\k \subset \g$, see Remark~\ref{RemConj}. To
prove Theorem~\ref{ThMain1}, we first exhibit various standard examples for
pairs of involutions on exceptional Lie algebras in
Section~\ref{Constructions}. Then it is shown in Section~\ref{Classification}
that these constructions already exhaust all possibilities for \ZZss.

\begin{theorem}\label{ThMain2}
Let $\k \subset \g=\so(8)$ be a subalgebra such that there are two commuting
involutions $\s$ and $\t$, $\s \neq \t$, of~$\g$ with $\k = \gs \cap \gt$. Then
there is an automorphism $\varphi$ of~$\so8$ such that $\varphi(\k)$ is one of
the following subalgebras of~$\so8$.
\begin{enumerate}

\item
$\so{n_1} + \so{n_2} + \so{n_3}$,\;\, $n_1 + n_2 + n_3 = 8$,\; $n_i \ge 1$;

\item
$\so{n_1} + \so{n_2} + \so{n_3} + \so{n_4}$,\;\, $n_1 + n_2 + n_3 + n_4 = 8$,\;
$n_i \ge 1$;

\item
$\u3 + \u1$.

\end{enumerate}
\end{theorem}

Note that for the case of $\so8$ there are less cases than there are in the
classification of Bahturin and Goze~\cite{bg} for the Lie algebras~$\so n, n
\neq 8$, since some of the $\ZZ$--symmetric subalgebras for general $\so n$ which
appear as distinct cases in~\cite{bg} are conjugate by some automorphism
of~$\so8$. Theorem~\ref{ThMain2} is proved in
Section~\ref{SpinClass}.

%%%%%%%%%%%%%%%%%%%%%%%%%%%%%%%%%%%%%%%%%%%%%%%%%%%%%%%%%%%%%%%%%%%%%%%%%%%%%%%%

\section{Preliminaries}
\label{Prelim}

%%%%%%%%%%%%%%%%%%%%%%%%%%%%%%%%%%%%%%%%%%%%%%%%%%%%%%%%%%%%%%%%%%%%%%%%%%%%%%%%

\begin{definition}\label{DefSymSub}
Let $\g$ be a Lie algebra. We say that $\k \subset \g$ is a {\em symmetric
subalgebra} of~$\g$ if there is a nontrivial automorphism~$\sigma$ of~$\g$
with $\sigma^2 = \id_{\g}$ such that $\k = \g^{\sigma} := \lbc X \in \g \mid
\sigma(X) = X \rbc$. If $G$ is a Lie group, we say that a closed subgroup $K
\subset G$ is a {\em symmetric subgroup} if the Lie algebra of~$K$ is a
symmetric subalgebra in the Lie algebra of~$G$.
\end{definition}

%%%%%%%%%%%%%%%%%%%%%%%%%%%%%%%%%%%%%%%%%%%%%%%%%%%%%%%%%%%%%%%%%%%%%%%%%%%%%%%%

\begin{proposition}\label{PropSplitting}
Let $\g$ be a real Lie algebra. Let $\sigma, \tau$ be automorphisms of~$\g$
such that $\sigma^2 = \tau^2 = \id_{\g}$, $\sigma \tau = \tau \sigma$, $\s \neq
\t$. Then the Lie algebra $\g$ splits as a direct sum of vector spaces
\begin{align}\label{EqSplitting}
\g = \g_1 \oplus \g_{\sigma} \oplus \g_{\tau} \oplus \g_{\sigma\tau},
\end{align}
where $\g_1 = \g^{\sigma} \cap \g^{\tau}$, $\g_1 + \g_{\sigma} = \g^{\sigma}$,
$\g_1 + \g_{\tau} = \g^{\tau}$, $\g_1 + \g_{\sigma\tau} = \g^{\sigma\tau}$ such
that the following hold:
\begin{enumerate}
\item
We have $\lbk \g_{\varphi}, \g_{\psi} \rbk \subseteq \g_{\varphi\psi}$ for all
$\varphi,\psi \in \Gamma$, i.e.\ the Lie algebra~$\g$ is $\Gamma$-graded, where
$\Gamma = \lbc 1, \sigma, \tau, \sigma\tau \rbc$ is the subgroup of~$\Aut(\g)$
generated by $\sigma$ and $\tau$.
\item
For all $\varphi \in \Gamma \setminus \lbc 1 \rbc$ we have that $\g_1$ is a
symmetric subalgebra of~$\g_1 + \g_{\varphi}$.
\item
For all $\varphi \in \Gamma \setminus \lbc 1 \rbc$ we have that $\g_1 +
\g_{\varphi}$ is a symmetric subalgebra of~$\g$.
\end{enumerate}
\end{proposition}

\begin{proof}
Consider the Cartan decomposition $\g = \g^{\s} + \pp^{\s}$. Since $\t$
commutes with $\s$, it leaves this decomposition invariant and hence we have
further splittings $\g^{\s} = (\g^{\s} \cap \g^{\t}) + (\g^{\s} \cap \pp^{\t})$
and $\pp^{\s} = (\pp^{\s} \cap \g^{\t}) + (\pp^{\s} \cap \pp^{\t})$, where $\g
= \g^{\t} + \pp^{\t}$ is the Cartan decomposition with respect to~$\t$. Define
$\g_{\s} = \g^{\s} \cap \pp^{\t}$, $\g_{\t} = \g^{\t} \cap \pp^{\s}$ and
$\g_{\s\t} = \pp^{\s} \cap \pp^{\t}$. Now the assertions of the proposition are
easily checked.
\end{proof}

Conversely, given a $\ZZ$--grading on a Lie algebra~$\g$, one can define a pair
of commuting involutions $\s,\t$ on~$\g$ such that (\ref{EqSplitting}) agrees
with the grading. Thus our results are equivalent to a classification of
$\ZZ$-gradings on $\Le_6, \Le_7, \Le_8, \Lf_4, \Lg_2$ and $\so8$.

%%%%%%%%%%%%%%%%%%%%%%%%%%%%%%%%%%%%%%%%%%%%%%%%%%%%%%%%%%%%%%%%%%%%%%%%%%%%%%%%

In order to classify \ZZss, we want to find all possibilities for a pair of
involutions $(\sigma, \tau)$ such that $\sigma$ and $\tau$ commute. It follows
from Proposition~\ref{PropSplitting} that in this case we have
\begin{align}\label{EqDimCond}
\dim \g & - \dim \g^{\s} - \dim \g^{\tau} = \dim \g^{\s\t} - 2 \dim \g^{\s}
\cap \g^{\t}.
\end{align}
This gives a necessary condition for the existence of \ZZss\ of certain types.
Note that the left hand side in~(\ref{EqDimCond}) depends only on the conjugacy
classes of $\s$ and $\t$; we define $d := d(\s,\t) := \dim \g  - \dim \g^{\s} -
\dim \g^{\tau}$. In order to find candidates for the pair $\lp \g^{\s\t},
\g^{\s} \cap \g^{\t} \rp$, we list all symmetric subalgebras of symmetric
subalgebras in exceptional compact Lie algebras in Table~\ref{TSymSym}, where
for each pair~$(\h,\k)$ the number $c(\h,\k) := \dim \h - 2 \dim \k$ is given.
By~(\ref{EqDimCond}), the only candidates for a pair $(\g^{\s\t}, \g^{\s} \cap
\g^{\t})$ are those pairs~$(\h,\k)$, where the number $c(\h,\k)$ equals~$d$.
Furthermore, we may eliminate immediately all pairs $(\h,\k)$ from the list of
candidates which do not fulfil the necessary condition that $\g^{\s}$ and
$\g^{\t}$ both contain a symmetric subalgebra isomorphic to~$\k$. However,
these are only necessary conditions and we have to check in each case if a
decomposition~(\ref{EqSplitting}) with $\g^{\s\t} \cong \h$, $\g^{\s} \cap
\g^{\t} \cong \k$ actually exists.

%%%%%%%%%%%%%%%%%%%%%%%%%%%%%%%%%%%%%%%%%%%%%%%%%%%%%%%%%%%%%%%%%%%%%%%%%%%%%%%%

For the convenience of the reader we list the symmetric subalgebras~$\k$ of
simple compact exceptional Lie algebras~$\g$ below in Table~\ref{TInv}; in
Tables~\ref{THermann} and \ref{THSpinAcht} we list pairs of (conjugacy classes
of) involutions on the exceptional Lie algebras and $\so8$, respectively, together
with the numbers $d = d(\s,\t)$.

%%%%%%%%%%%%%%%%%%%%%%%%%%%%%%%%%%%%%%%%%%%%%%%%%%%%%%%%%%%%%%%%%%%%%%%%%%%%%%%%
\begin{table}[!h]\rm
\begin{tabular}{|l|c|c|}
\hline \str & $\g$ & $\k$  \\ \hline

 E\,I & $\Le_6$ & $\sp4$ \\

 E\,II & $\Le_6$ & $\su6 + \sp1$ \\

 E\,III & $\Le_6$ & $\so(10) + \R$ \\

 E\,IV & $\Le_6$ & $\Lf_4$ \\ \hline

 E\,V & $\Le_7$ & $\su8$ \\

 E\,VI & $\Le_7$ & $\so(12) + \sp1$ \\ \hline

\end{tabular}\qquad
\begin{tabular}{|l|c|c|}
\hline  \str & $\g$ & $\k$  \\ \hline

 E\,VII & $\Le_7$ & $\Le_6 + \R$ \\ \hline

 E\,VIII & $\Le_8$ & $\so(16)$ \\

 E\,IX & $\Le_8$ & $\Le_7 + \sp1$ \\ \hline

 F\,I & $\Lf_4$ & $\sp3 + \sp1$ \\

 F\,II & $\Lf_4$ & $\so9$ \\ \hline

 G & $\Lg_2$ & $\sp1 + \sp1$ \\ \hline

\end{tabular}
\bl

\caption{Exceptional symmetric spaces of type~I.} \label{TInv}

\end{table}

%%%%%%%%%%%%%%%%%%%%%%%%%%%%%%%%%%%%%%%%%%%%%%%%%%%%%%%%%%%%%%%%%%%%%%%%%%%%%%%%

\section{Involutions of reductive Lie algebras and construction of examples}
\label{Constructions}

%%%%%%%%%%%%%%%%%%%%%%%%%%%%%%%%%%%%%%%%%%%%%%%%%%%%%%%%%%%%%%%%%%%%%%%%%%%%%%%%

We will now discuss some constructions of \ZZss. We start with a list of
several possibilities how one can define involutive automorphisms of a
reductive complex Lie algebra from a given root space decomposition. Let $\g$
be a reductive complex Lie algebra with Cartan subalgebra~$\g_0$ and let
\begin{equation}\label{EqRootDec}
\g = \g_0 + \sum_{\alpha \in \Phi} \g_{\alpha},
\end{equation}
be the root space decomposition with respect to $\g_0$. Also assume that we
have chosen once and for all a set $\lbc \a_1, \ldots, \a_n \rbc$ of simple
roots.

\paragraph{\bf Type 1}
Let $\h \subset \g$ be a symmetric subalgebra of maximal rank. Then we may
assume that $\h$ contains~$\g_0$. Hence there is a subset $S \subset \Phi$ such
that $\h = \g_0 + \sum_{\alpha \in S} \g_{\alpha}$ and we may define an
involutive automorphism $\s$ of~$\g$ by requiring that $\s(X)=X$ if $X \in \h$
and $\s(X)=-X$ if $X \in \sum_{\alpha \in \Phi \setminus S} \g_{\alpha}$. The
corresponding symmetric spaces are exactly those where the involution is an
inner automorphism.

\paragraph{\bf Type 2}
Outer involutions induced from automorphisms of the Dynkin diagram,
see~\cite{helgason}, Ch.~X, \S~5: The corresponding simply connected
irreducible symmetric spaces with simple compact isometry group are
$\SU(2n+1)/\SO(2n+1)$, $\SU(2n)/\Sp(n)$, $\SO(2n+2)/\SO(2n+1)$ and $\LE_6 /
\LF_4$.

\paragraph{\bf Type 3} Let $\g$ be a reductive complex Lie algebra and let
$\tt \subset \g$ be a Cartan subalgebra. Then there is an automorphism~$\s$
of~$\g$ which acts as minus identity on~$\tt$ and sends each root to its
negative. The symmetric spaces~$G/H$ given in this way are exactly those where
$\rk G/H = \rk G$; the Satake diagram of $G/H$ is given by the Dynkin diagram
of~$G$ with uniform multiplicity one. The automorphisms obtained in this way
may be inner or outer. The corresponding irreducible compact symmetric space
are: A\,I, $\SO(2n+1) / \SO(n+1) \times \SO(n)$, $\Sp(n) / \U(n)$, $\SO(2n) /
\SO(n) \times \SO(n)$, E\,I, E\,V, E\,VIII, F\,I, G.

Note that the distinction of the three types of involutions above pertains to
the action of the involution with respect to one fixed Cartan subalgebra. While
Type~1 involutions obviously are never conjugate to involutions of type Type~2,
a Type~3 involution may be conjugate to a Type~1 or Type~2 involution.

In the examples below, we show how to construct various \ZZss\ mainly by combining two
commuting involutions as defined above with respect to one fixed root space
decomposition~(\ref{EqRootDec}). We start with the examples where both
involutions are of Type~1.

\begin{example}\label{Eg1}\rm
Let $G$ be a connected compact Lie group and let $\s_1$ and $\s_2$ be two
involutions of~$G$ defined by conjugation with the elements $g_1$ and $g_2$,
respectively, of $G$, such that the fixed point sets $G^{\s_1}$ and $G^{\s_2}$
are non-isomorphic. After conjugation, we may assume that $g_1$ and $g_2$ are
both contained in one and the same maximal Torus~$T$ of~$G$; in particular,
$\s_1$ and $\s_2$ commute. Consider the root space
decomposition~(\ref{EqRootDec}) where $\g_0$ is the complexification of the Lie
algebra of~$T$. Then the automorphisms of $\gc$ induced by $\s$ and $\t$ are
both of Type~1 and the complexification of the Lie algebra of $G^{{\s}} \cap
G^{{\t}}$ is given by $\g_0 + \sum_{\alpha \in S_1 \cap S_2} \g_{\alpha}$,
where $S_i$ is the root system of $G^{\s_i}$, $i=1,2$, w.r.t.\ $\g_0$.

The maximal subgroups of maximal rank in a simple compact Lie group~$G$ can be
conveniently described by using extended Dynkin diagrams,
cf.~\cite{oniscikBook}, 1.3.11. The nodes of the extended Dynkin diagram
correspond to roots $\a_0, \a_1,\ldots, \a_n$, where $\a_1, \ldots, \a_n$ are
the simple roots and $\a_0 = -\delta$ where $\delta$ is highest root. The
simple roots of a maximal subgroup of maximal rank are then given by certain
subsets of the set of nodes of the extended Dynkin diagram. For the groups
$\LE_6$, $\LE_7$ and $\LE_8$ these subsets are given below by the
black nodes~\hbox{{\put(5,3){\circle*{3}}}{\ \ \,}} for the symmetric ones
among the subgroups of maximal rank. The root systems $S_i$ of these subgroups
are given as the union of all roots which are integral linear combinations of
the simple roots corresponding to black nodes. Now we may consider two diagrams
depicting two symmetric subgroups~$G^{\s_1}$ and $G^{\s_2}$ of
maximal rank. Then the root system of~$G^{\s_1} \cap G^{\s_2}$ with respect to
$\g_0$ is given by $S_1 \cap S_2$ and it is a straightforward task to
explicitly determine this set. However, to completely avoid these computations, we will
use a simplified approach. We know from \cite{oniscikBook}, 1.3,
Proposition~15, that both $G^{\s_1}$ and $G^{\s_2}$ contain a subgroup whose
simple roots with respect to $\g_0$ are those whose corresponding nodes are
marked black in both diagrams. Thus we obtain the set of simple roots of a
regular subgroup of~$G$ contained in the intersection $G^{\s} \cap G^{\t}$.
This information, together with a dimension count using Tables~\ref{THermann}
and \ref{TSymSym}, turns out to be sufficient in all cases given below to
identify the connected component of the intersection $G^{\s} \cap G^{\t}$,
which is a symmetric subgroup of both $G^{\s}$ and $G^{\t}$.

The symmetric subgroups of maximal rank in $\LE_6$, namely $\SU6 \cdot \Sp1$
and $\Spin(10) \cdot \U1$ are given by the following diagrams:
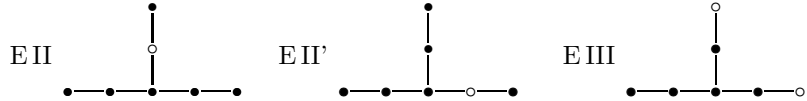
\begin{center}
E\,II\begin{minipage}{80pt}
\begin{picture}(100,60)
\multiput(6,17)(16,0){5}{\circle*{3}} \multiput(6,17)(16,0){5}{\circle*{3}}
\multiput(8.5,17)(16,0){4}{\line(8,0){11}}
\multiput(38,19.5)(0,16){2}{\line(0,8){11}} \put(38,33){\circle{3}}
\put(38,49){\circle*{3}} \put(4,5){$$} \put(20,5){$$} \put(36,5){$$}
\put(52,5){$$} \put(68,5){$$} \put(44,30){$$}
\end{picture}
\end{minipage}\;\;E\,II'\begin{minipage}{80pt}
\begin{picture}(100,60)
\multiput(6,17)(16,0){5}{\circle{3}} \multiput(6,17)(16,0){3}{\circle*{3}}
\multiput(8.5,17)(16,0){4}{\line(8,0){11}}
\multiput(38,19.5)(0,16){2}{\line(0,8){11}}
\put(38,33){\circle*{3}}\put(70,17){\circle*{3}} \put(38,49){\circle*{3}}
\put(4,5){$$} \put(20,5){$$} \put(36,5){$$} \put(52,5){$$} \put(68,5){$$}
\put(44,30){$$}
\end{picture}
\end{minipage}
\;\;E\,III\begin{minipage}{80pt}
\begin{picture}(100,60)
\multiput(6,17)(16,0){4}{\circle*{3}} \multiput(6,17)(16,0){5}{\circle{3}}
\multiput(8.5,17)(16,0){4}{\line(8,0){11}}
\multiput(38,19.5)(0,16){2}{\line(0,8){11}} \put(38,33){\circle*{3}}
\put(38,33){\circle{3}} \put(38,49){\circle{3}} \put(4,5){$$} \put(20,5){$$}
\put(36,5){$$} \put(52,5){$$} \put(68,5){$$} \put(44,30){$$}
\end{picture}
\end{minipage}
\end{center}
Comparing E-II and E-III, we see that there is a \ZZs\ $G/K$ such that $\k$
contains a subalgebra isomorphic to $\su5 + \R + \R$. A simple dimension count,
facilitated by Tables~\ref{THermann} and \ref{TSymSym}, shows that the
corresponding \ZZs\ is of type E\,II-III-III with $\k \cong \su5 + \R + \R$.
Using an alternative embedding of $\SU6 \cdot \Sp1$ into $\LE_6$, shown as
E\,II', and combining it with E\,II, we obtain a \ZZs\ with $\k \cong \su4 +
\sp1 + \sp1 + \R \cong \so6 + \so4 + \R$ of type E\,II-II-III. Combining E\,III
and E\,III' (as given in Example~\ref{Eg5} below) we construct a \ZZs\ of type
E\,III-III-III with $\k = \so8 + \R + \R$.

The connected symmetric subgroups of~$\LE_7$, namely $\SU(8)$,
$\Spin(12)\cdot\Sp(1)$ and $\LE_6 \cdot \U1$, are given by the following
diagrams.
\begin{center}
E\,V\begin{minipage}{108pt}
\begin{picture}(100,45)
\multiput(22,17)(16,0){6}{\circle*{3}} \multiput(6,17)(16,0){6}{\circle*{3}}
\multiput(8.5,17)(16,0){6}{\line(8,0){11}}
\multiput(54,19.5)(0,16){1}{\line(0,8){11}}
\multiput(102,17)(16,0){1}{\circle*{3}} \put(54,33){\circle{3}} \put(4,5){$$}
\put(20,5){$$} \put(36,5){$$} \put(52,5){$$} \put(68,5){$$} \put(44,30){$$}
\end{picture}
\end{minipage}\hspace{2em}
E\,VI\begin{minipage}{108pt}
\begin{picture}(100,45)
\multiput(6,17)(16,0){5}{\circle*{3}} \multiput(6,17)(16,0){7}{\circle{3}}
\multiput(8.5,17)(16,0){6}{\line(8,0){11}}
\multiput(54,19.5)(0,16){1}{\line(0,8){11}} \put(54,33){\circle*{3}}
\put(102,17){\circle*{3}} \put(4,5){$$} \put(20,5){$$} \put(36,5){$$}
\put(52,5){$$} \put(68,5){$$} \put(44,30){$$}
\end{picture}
\end{minipage}\hspace{2em}
E\,VI'\begin{minipage}{108pt}
\begin{picture}(100,45)
\multiput(38,17)(16,0){5}{\circle*{3}} \multiput(6,17)(16,0){7}{\circle{3}}
\multiput(8.5,17)(16,0){6}{\line(8,0){11}}
\multiput(54,19.5)(0,16){1}{\line(0,8){11}} \put(54,33){\circle*{3}}
\put(6,17){\circle*{3}} \put(4,5){$$} \put(20,5){$$} \put(36,5){$$}
\put(52,5){$$} \put(68,5){$$} \put(44,30){$$}
\end{picture}
\end{minipage}\hspace{2em}
E\,VII\begin{minipage}{108pt}
\begin{picture}(100,45)
\multiput(22,17)(16,0){5}{\circle*{3}} \multiput(6,17)(16,0){5}{\circle{3}}
\multiput(8.5,17)(16,0){6}{\line(8,0){11}}
\multiput(54,19.5)(0,16){1}{\line(0,8){11}}
\multiput(102,17)(16,0){1}{\circle{3}} \put(54,33){\circle{3}}
\put(54,33){\circle*{3}} \put(4,5){$$} \put(20,5){$$} \put(36,5){$$}
\put(52,5){$$} \put(68,5){$$} \put(44,30){$$}
\end{picture}
\end{minipage}
\end{center}
Combining E\,V and E\,VI we obtain a \ZZs\ with $\k = \su6 + \sp1 +
 \R$ of type E\,V-VI-VII. Combining E\,VI and
E\,VII we obtain a \ZZs\ with $\k$ containing a subalgebra isomorphic to
$\so{10} + \R + \R$; using Tables~\ref{THermann} and \ref{TSymSym} it follows
that the corresponding \ZZs\ is of type E\,VI-VII-VII with $\k \cong \so{10} +
\R + \R$. Combining E\,VI with E\,VI', we obtain a \ZZs\ such that $\k$
contains a subalgebra isomorphic to $\so8 + \sp1 + \sp1$. We see from
Tables~\ref{THermann} and \ref{TSymSym} that the corresponding \ZZs\ is of type
E\,VI-VI-VI with $\k \cong \so8 + \so4 + \sp1$.

The connected symmetric subgroups of~$\LE_8$, namely ${\rm SO'}(16)$ and $\LE_7
\cdot \Sp1$ are given by the following diagrams.
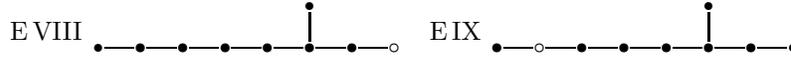
\begin{center}
E\,VIII\begin{minipage}{108pt}
\begin{picture}(100,45) \multiput(22,17)(16,0){7}{\circle{3}}
\multiput(6,17)(16,0){6}{\circle*{3}}
\multiput(8.5,17)(16,0){7}{\line(8,0){11}}
\multiput(86,19.5)(0,16){1}{\line(0,8){11}}
\multiput(102,17)(16,0){1}{\circle*{3}} \put(86,33){\circle*{3}} \put(4,5){$$}
\put(20,5){$$} \put(36,5){$$} \put(52,5){$$} \put(68,5){$$} \put(44,30){$$}
\end{picture}
\end{minipage}\hspace{2em}
E\,IX\begin{minipage}{108pt}
\begin{picture}(100,45)
\multiput(6,17)(16,0){7}{\circle{3}} \multiput(6,17)(16,0){1}{\circle*{3}}
\multiput(38,17)(16,0){6}{\circle*{3}}
\multiput(8.5,17)(16,0){7}{\line(8,0){11}}
\multiput(86,19.5)(0,16){1}{\line(0,8){11}}
\multiput(102,17)(16,0){1}{\circle*{3}} \put(86,33){\circle*{3}} \put(4,5){$$}
\put(20,5){$$} \put(36,5){$$} \put(52,5){$$} \put(68,5){$$} \put(44,30){$$}
\end{picture}
\end{minipage}
\end{center}
Combining the two diagrams and using Tables~\ref{THermann} and \ref{TSymSym}
proves the existence of a \ZZs\ of type E\,VIII-IX-IX with $\k \cong \so{12} +
\sp1 +\sp1$.
\end{example}

%%%%%%%%%%%%%%%%%%%%%%%%%%%%%%%%%%%%%%%%%%%%%%%%%%%%%%%%%%%%%%%%%%%%%%%%%%%%%%%%

\begin{example}\label{Eg2}\rm
Let now $G$ be simple and let $g \in G$ be an element with $i_g^2 = \id_G$.
Let $T$ be a maximal torus of~$G$ containing~$g$. Let $w=k \, Z_G(T)$, $k \in
N_G(T)$, be an element of the Weyl group $W_G = N_G(T) / Z_G(T)$ of~$G$ such
that the action of~$w$ on~$T$ does not leave the root system of $Z_G(\lbc g
\rbc)$ invariant. Define $h = w \act g = k\,g\,k^{-1}$. Then $i_g$ and $i_h$
are two conjugate commuting involutions of~$G$ whose fixed point sets do not
agree. Hence they define a $\ZZ$--symmetric space.

We will now use this construction to prove the existence of various \ZZss.
First, let $G = \LF_4$ and let $i_g$ be an inner involution of type F\,II,
i.e.\ the connected component of the fixed point set of $i_g$ is isomorphic to
$\Spin9$. Define another involution $i_h$ as above; it is also of type F\,II,
but the fixed point sets of $i_g$ and $i_h$ do not agree. From
Table~\ref{THermann} we read off that we have $d = -20$ for a pair of
involutions of Type F\,II-II. In Table~\ref{TSymSym} there are two entries with
$\g = \Lf_4$ and $c(\h,\k)= -20$, namely $\k = \sp3 + \R$ and $\k = \so8$.
Since $\sp3$ is not subalgebra of $\so9$, it follows that the \ZZs\ constructed
in this way is of type F\,II-II-II. This argument also shows that there is no
\ZZs\ of type F\,I-II-II.

Now let $G = \LE_8$ and let $\s$ be an involution of type E\,IX as given in
Example~\ref{Eg1} and let $\g^{\s}$ be the Lie algebra of the fixed point set
of $i_g$. Consider the subalgebra~$\h \subset \g=\Le_8$ isomorphic to $\su3 +
\Le_6$ given by the diagram below.
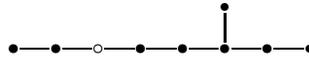
\begin{center}
\begin{minipage}{108pt}
\begin{picture}(100,45)
\multiput(6,17)(16,0){7}{\circle{3}} \multiput(6,17)(16,0){2}{\circle*{3}}
\multiput(54,17)(16,0){5}{\circle*{3}}
\multiput(8.5,17)(16,0){7}{\line(8,0){11}}
\multiput(86,19.5)(0,16){1}{\line(0,8){11}}
\multiput(102,17)(16,0){1}{\circle*{3}} \put(86,33){\circle*{3}} \put(4,5){$$}
\put(20,5){$$} \put(36,5){$$} \put(52,5){$$} \put(68,5){$$} \put(44,30){$$}
\end{picture}
\end{minipage}
\end{center}
Let $W$ be the Weyl group of~$\Le_8$ with respect to $\g_0 = \tt$. The Weyl
groups of $\g^{\s}$ and of $\h$ are subgroups of $W$ in a natural way.
Comparing the two diagrams corresponding to $\g^{\s}$ and $\h$, respectively,
we see that there is at least one root space of $\g$ which is contained in the
$\su3$-summand of $\h$, but not in $\g^{\s}$. Hence there is a Weyl group
element $w \in W \subset \Aut(\g)$ whose action on~$\g$ leaves the
$\Le_6$-summand of~$\h$ fixed, but which does not leave $\gs$ invariant. This
implies that there is an involution $\t \neq \s$ such that $\gs \cap \gt$
contains a subalgebra isomorphic to $\Le_6$. By a dimension count, this shows
the existence of a \ZZs\ of type E\,IX-IX-IX with $\k \cong \Le_6 + \R + \R$.

Now consider $\g = \Le_7$ and let $\s$ be an involution as given by the diagram
E\,VI in Example~\ref{Eg1}. Consider the subalgebra $\h \subset \g$ isomorphic
to $\su6 + \su3$ given by the following diagram.
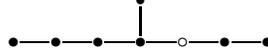
\begin{center}
\begin{minipage}{108pt}
\begin{picture}(100,45)
\multiput(6,17)(16,0){4}{\circle*{3}} \multiput(6,17)(16,0){7}{\circle{3}}
\multiput(8.5,17)(16,0){6}{\line(8,0){11}}
\multiput(54,19.5)(0,16){1}{\line(0,8){11}} \put(54,33){\circle*{3}}
\put(102,17){\circle*{3}}\put(86,17){\circle*{3}} \put(4,5){$$} \put(20,5){$$}
\put(36,5){$$} \put(52,5){$$} \put(68,5){$$} \put(44,30){$$}
\end{picture}
\end{minipage}
\end{center}
With an analogous argument as above we may construct a \ZZs\ such that $\k$
contains a subalgebra isomorphic to $\su6$. Using Table~\ref{TSymSym} we see
that the corresponding \ZZs\ is of type E\,VI-VI-VI with $\k \cong \u6 +  \R$.

For $G = \LE_6$, consider the subgroup locally isomorphic to $\SU3 \cdot \SU3
\cdot \SU3$ given by the following diagram.
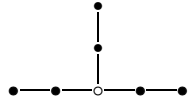
\begin{center}
\begin{minipage}{108pt}
\begin{picture}(100,52)
\multiput(6,17)(16,0){5}{\circle{3}} \multiput(6,17)(16,0){2}{\circle*{3}}
\multiput(8.5,17)(16,0){4}{\line(8,0){11}}
\multiput(38,19.5)(0,16){2}{\line(0,8){11}}
\put(54,17){\circle*{3}}\put(38,33){\circle*{3}}\put(70,17){\circle*{3}}
\put(38,49){\circle*{3}} \put(4,5){$$} \put(20,5){$$} \put(36,5){$$}
\put(52,5){$$} \put(68,5){$$} \put(44,30){$$}
\end{picture}
\end{minipage}
\end{center}
Lt $i_g$ be an inner involution corresponding to the diagram~E\,II from
Example~\ref{Eg1}. Then there is a Weyl group element $w = k \, Z_G(T)$ such
that the involution $i_h$ with $h= k \, g \, k^{-1}$ is such that the Lie
algebra of the common fixed point set of $i_g$ and $i_h$ contains a subalgebra
isomorphic to $\su3 + \su3 + \R + \R$. It follows from Table~\ref{TSymSym} that
we have constructed a \ZZs\ of type E\,II-II-II with $\k = \su3 + \su3 + \R +
\R$.
\end{example}

%%%%%%%%%%%%%%%%%%%%%%%%%%%%%%%%%%%%%%%%%%%%%%%%%%%%%%%%%%%%%%%%%%%%%%%%%%%%%%%%

\begin{example}\label{Eg3}\rm
Here we combine two involutions of Type~3 and Type~1, respectively. Thus let
$\g$ be a complex semisimple Lie algebra and let $\g_0 \subset \g$ be a Cartan
subalgebra. Let $\s$ be an involution of Type~3 on $\g$. Let $\t$ be an inner
involution given as conjugation by an element $t = \exp X$, where $X \in \g_0$.
Then the connected component of $Z_G(\lbc t \rbc)$ is the connected Lie
subgroup of $G$ corresponding to $\g_0 + \sum_{\alpha \in S} \g_{\alpha}$,
where $S$ is a certain subset of the root system~$\Phi$ of $\g$ with the
property that for each $\alpha \in S$ we have also $-\alpha \in S$. Obviously,
the two involutions commute and the involution induced by $\s$ on the fixed
point set~$\g^{\t}$ of $\t$ is the automorphism which acts as minus identity
on~$\g_0$ and sends each root $\alpha \in S$ of $\g^{\t}$ to its negative.
Moreover, the fixed point set of $\s\t$ will be isomorphic to the fixed point
set of $\s$, since the corresponding symmetric spaces will have isomorphic
Satake diagrams. Hence the data of the \ZZs\ obtained in this way can be
immediately deduced from the classification of ordinary symmetric spaces. This
shows the existence of the exceptional \ZZs\ of types E\,I-I-II, E\,I-I-III,
E\,V-V-V, E\,V-V-VI, E\,V-V-VII, E\,VIII-VIII-VIII, E\,VIII-VIII-IX, F\,I-I-I,
F\,I-I-II, G as given in Table~\ref{TExcSpaces}.
\end{example}

%%%%%%%%%%%%%%%%%%%%%%%%%%%%%%%%%%%%%%%%%%%%%%%%%%%%%%%%%%%%%%%%%%%%%%%%%%%%%%%%

\begin{example}\label{Eg4}\rm
Let now $G = \LE_6$ and choose a Cartan subalgebra $\g_0 \subset \Le_6$. Let
$\s$ be an involution of Type~2, i.e.\ E\,IV and let $\t$ be an involution of
Type~3, i.e.\ E\,I. Then the two involutions obviously commute. A dimension count using
Tables~\ref{THermann} and \ref{TSymSym} shows that $\s\t$ is of type E\,II and hence
the corresponding \ZZs\ is of type E\,I-II-IV with $\k \cong \sp3 + \sp1$.
\end{example}

%%%%%%%%%%%%%%%%%%%%%%%%%%%%%%%%%%%%%%%%%%%%%%%%%%%%%%%%%%%%%%%%%%%%%%%%%%%%%%%%

\begin{example}\label{Eg6}\rm
Let $G = \LE_6$. Let $\s$ be the involution of Type~2 given by the permutation
of the simple roots $\a_1 \mapsto \a_5$, $\a_2 \mapsto \a_4$, $\a_3 \mapsto
\a_3$, $\a_4 \mapsto \a_2$, $\a_5 \mapsto \a_1$, $\a_6 \mapsto \a_6$ and let
$\t$ be the involution of Type~1 given by the following diagram, where the
numbering of the roots $\a_0, \ldots, \a_6$ is given.
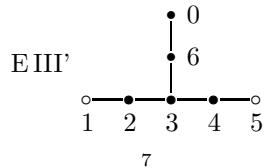
\begin{center}
E\,III'\begin{minipage}{80pt}
\begin{picture}(100,60)
\multiput(22,17)(16,0){3}{\circle*{3}} \multiput(6,17)(16,0){5}{\circle{3}}
\multiput(8.5,17)(16,0){4}{\line(8,0){11}}
\multiput(38,19.5)(0,16){2}{\line(0,8){11}} \put(38,33){\circle*{3}}
\put(38,49){\circle*{3}} \put(4,5){$1$} \put(20,5){$2$} \put(36,5){$3$}
\put(52,5){$4$} \put(68,5){$5$} \put(44,30){$6$} \put(44,46){$0$}
\end{picture}
\end{minipage}
\end{center}
Then obviously $\s$ and $\t$ commute and since $\s\t$ is an outer automorphism
of $\LE_6$, it follows from Lemma~\ref{LmNonEx} that we have constructed a
\ZZs\ of type E\,III-IV-IV with $\k \cong \so9$.
\end{example}

%%%%%%%%%%%%%%%%%%%%%%%%%%%%%%%%%%%%%%%%%%%%%%%%%%%%%%%%%%%%%%%%%%%%%%%%%%%%%%%%

\begin{example}\label{Eg5}\rm
Let $G = \LE_7$. By \cite{wolfIrr}, Thm.~3.1, there is an isotropy irreducible
homogeneous space $\LE_7 / L$, where $L$ is locally isomorphic to $\Sp1 \cdot \LF_4$
such that the $78$-dimensional isotropy representation~$\rho$ is
equivalent to the tensor product of the adjoint representation of $\Sp1$ and
the $26$-dimensional irreducible representation of~$\LF_4$. In particular, we
may view the isotropy representation as a representation of $\SO3 \times \LF_4$
and choose two elements $g,h \in L$ such that the action of $\rho(g)$ and
$\rho(h)$ is given by $\diag(-1,-1,+1) \otimes \e$ and $\diag(1,-1,-1) \otimes
\e$, respectively, where $\e$ stands for the identity element of~$\LF_4$. Then
the involutions $i_g$ and $i_h$ generate a subgroup $\Gamma \subset
\Aut(\LE_7)$ isomorphic to $\ZZ$ whose fixed point set has Lie algebra $\k
\cong \Lf_4$. Since the fixed point sets of $i_g$, $i_h$ and $i_{gh}$ are all
$79$-dimensional, it follows that we have constructed a \ZZs\ of type
E\,VII-VII-VII.
\end{example}

%%%%%%%%%%%%%%%%%%%%%%%%%%%%%%%%%%%%%%%%%%%%%%%%%%%%%%%%%%%%%%%%%%%%%%%%%%%%%%%%

\section{The classification in the exceptional case}
\label{Classification}

%%%%%%%%%%%%%%%%%%%%%%%%%%%%%%%%%%%%%%%%%%%%%%%%%%%%%%%%%%%%%%%%%%%%%%%%%%%%%%%%
\begin{table}[t]\rm
\begin{tabular}{|l|c|}
\hline \str Type & $d$  \\
\hline \hline

 E\,I-I & $6$ \\ \hline

 E\,I-II & $4$ \\ \hline

 E\,I-III & $-4$ \\ \hline

 E\,I-IV & $-10$ \\ \hline

 E\,II-II & $2$ \\ \hline

 E\,II-III & $-6$ \\ \hline

 E\,II-IV\;\;\;\;\;\;\; & $-12$ \\ \hline

\end{tabular}
\begin{tabular}{|l|c|}
\hline

 E\,III-III & $-14$ \\ \hline

 E\,III-IV & $-20$ \\ \hline

 E\,IV-IV & $-26$ \\ \hline

 E\,V-V & $7$ \\ \hline

 E\,V-VI & $1$ \\ \hline

 E\,V-VII & $-9$ \\ \hline

 E\,VI-VI & $-5$ \\ \hline

 E\,VI-VII\;\;\;\,\, & $-15$ \\ \hline

\end{tabular}
\begin{tabular}{|l|c|}
\hline
 E\,VII-VII & $-25$ \\ \hline

 E\,VIII-VIII & $8$ \\ \hline

 E\,VIII-IX & $-8$ \\ \hline

 E\,IX-IX & $-24$ \\ \hline

 F\,I-I & $4$ \\ \hline

 F\,I-II & $-8$ \\ \hline

 F\,II-II & $-20$ \\ \hline

 G & $2$ \\ \hline

\end{tabular}
\bl

\caption{Pairs of involutions on exceptional groups.} \label{THermann}

\end{table}
%%%%%%%%%%%%%%%%%%%%%%%%%%%%%%%%%%%%%%%%%%%%%%%%%%%%%%%%%%%%%%%%%%%%%%%%%%%%%%%%

%%%%%%%%%%%%%%%%%%%%%%%%%%%%%%%%%%%%%%%%%%%%%%%%%%%%%%%%%%%%%%%%%%%%%%%%%%%%%%%%
\begin{table}[!h]\rm
\begin{tabular}{|c|c|c|c|}
\hline \str $\g$ & $\h$ & $\k$ & $c(\h,\k)$   \\
\hline \hline

%-----------------------------------------------------------------

 $\Le_6$ & $\sp4$ &  $\sp(4-n) + \sp n$  & $-12,-4$ \\
 $$ & $$ &  $n=1,2$  & $$ \\ \hline

 $\Le_6$ & $\sp4$ &  $\su4 + \R$  & $4$ \\ \hline

 $\Le_6$ & $\su6 + \sp1$ &  $\suxu{6{-}n}n + \R$  & $-34,-14,$ \\
 $$ & $$ &  $n=0,\ldots,3$  & $-2,2$ \\ \hline

 $\Le_6$ & $\su6 + \sp1$ &  $\suxu{6{-}n}n + \sp1$  & $-18,-6,-2$ \\
 $$ & $$ &  $n=1,\ldots,3$  & $$ \\ \hline

 $\Le_6$ & $\su6 + \sp1$ &  $\sp3 + \sp1$  & $-10$ \\ \hline

 $\Le_6$ & $\su6 + \sp1$ &  $\sp3 + \R$  & $-6$ \\ \hline

 $\Le_6$ & $\su6 + \sp1$ &  $\so6 + \sp1$  & $2$ \\ \hline

 $\Le_6$ & $\su6 + \sp1$ &  $\so6 + \R$  & $6$ \\ \hline

 $\Le_6$ & $\so(10) + \R$ &  $\so(10{-}n) + \so n$  & $-44,-26,-12,$ \\
 $$ & $$ &  $n=0,\ldots,5$  & $-2,4,6$ \\ \hline

 $\Le_6$ & $\so(10) + \R$ &  $\so(10{-}n) + \so n + \R$  & $-28,-14,$ \\
 $$ & $$ &  $n=1,\ldots,5$  & $-4,2,4$ \\ \hline

 $\Le_6$ & $\so(10) + \R$ &  $\u5 + \R$  & $-6$ \\ \hline

 $\Le_6$ & $\so(10) + \R$ &  $\u5$  & $-4$ \\ \hline

 $\Le_6$ & $\Lf_4$ &  $\sp3 + \sp1$  & $4$ \\ \hline

 $\Le_6$ & $\Lf_4$ &  $\so9$  & $-20$ \\ \hline \hline
%

%-----------------------------------------------------------------

 $\Le_7$ & $\su8$ &  $\suxu{8-n}n$  & $-35,-15,$ \\
 $$ & $$ &  $n=1,\ldots,4$  & $-3,1$ \\ \hline

 $\Le_7$ & $\su8$ &  $\sp4$  & $-9$ \\ \hline

 $\Le_7$ & $\su8$ &  $\so8$  & $7$ \\ \hline

 $\Le_7$ & $\so(12) + \sp1$ &  $\so(12{-}n) + \so n + \R$  & $-65,-43,-25,$ \\
 $$ & $$ &  $n=0,\ldots,6$  & $-11,-1,5,7$ \\ \hline

 $\Le_7$ & $\so(12) + \sp1$ &  $\so(12{-}n) + \so n + \sp1,\, $  & $-47,-29,$ \\
 $$ & $$ &  $n=1,\ldots,6$  & $-15,-5,1,3$ \\ \hline

\end{tabular}
\bl \caption{Symmetric subalgebras of symmetric subalgebras of simple
exceptional compact Lie algebras.} \label{TSymSym}
\end{table}
\begin{table}
\begin{tabular}{|c|c|c|c|} \hline

 $\Le_7$ & $\so(12) + \sp1$ &  $\u6 + \sp1$  & $-9$ \\ \hline

 $\Le_7$ & $\so(12) + \sp1$ &  $\u6 + \R$  & $-5$ \\ \hline

 $\Le_7$ & $\Le_6 + \R$ &  $\sp4 + \R$  & $5$ \\ \hline

 $\Le_7$ & $\Le_6 + \R$ &  $\sp4$  & $7$ \\ \hline

 $\Le_7$ & $\Le_6 + \R$ &  $\so(10) + \R + \R$  & $-15$ \\ \hline

 $\Le_7$ & $\Le_6 + \R$ &  $\so(10) + \R$  & $-13$ \\ \hline

 $\Le_7$ & $\Le_6 + \R$ &  $\su6 + \sp1 + \R$  & $1$ \\ \hline

 $\Le_7$ & $\Le_6 + \R$ &  $\su6 + \sp1$  & $3$ \\ \hline

 $\Le_7$ & $\Le_6 + \R$ &  $\Lf_4 + \R$  & $-27$ \\ \hline

 $\Le_7$ & $\Le_6 + \R$ &  $\Lf_4$  & $-25$ \\ \hline \hline

 $\Le_8$ & $\so(16)$ &  $\so{16{-}n} + \so n$  & $-90,-64,-42,$ \\
 $$ & $$ &  $n=1,\ldots,8$  & $-24,-10,0,6,8$ \\ \hline

 $\Le_8$ & $\so(16)$ &  $\u8$  & $-8$ \\ \hline

 $\Le_8$ & $\Le_7 + \sp1$ &  $\Le_7 + \R$  & $-132$ \\ \hline

 $\Le_8$ & $\Le_7 + \sp1$ &  $\su8 + \sp1$  & $4$ \\ \hline

 $\Le_8$ & $\Le_7 + \sp1$ &  $\su8 + \R$  & $8$ \\ \hline

 $\Le_8$ & $\Le_7 + \sp1$ &  $\so(12) + \sp1 + \sp1$  & $-8$ \\ \hline

 $\Le_8$ & $\Le_7 + \sp1$ &  $\so(12) + \sp1 + \R$  & $-4$ \\ \hline

 $\Le_8$ & $\Le_7 + \sp1$ &  $\Le_6 + \R + \sp1$  & $-28$ \\ \hline

 $\Le_8$ & $\Le_7 + \sp1$ &  $\Le_6 + \R + \R$  & $-24$ \\ \hline \hline
%

%-----------------------------------------------------------------

 $\Lf_4$ & $\sp3 + \sp1$ &  $\sp3 + \R$  & $-20$ \\ \hline

 $\Lf_4$ & $\sp3 + \sp1$ &  $\sp2 + \sp1 + \sp1$  & $-8$ \\ \hline

 $\Lf_4$ & $\sp3 + \sp1$ &  $\sp2 + \sp1 + \R$  & $-4$ \\ \hline

 $\Lf_4$ & $\sp3 + \sp1$ &  $\u3 + \sp1$  & $0$ \\ \hline

 $\Lf_4$ & $\sp3 + \sp1$ &  $\u3 + \R$  & $4$ \\ \hline

 $\Lf_4$ & $\so9$ &  $\so{9{-}n} + \so n$  & $-20,-8,0,4$ \\
 $$ & $$ &  $n=1,\ldots,4$  & $$ \\ \hline \hline
%

%-----------------------------------------------------------------

 $\Lg_2$ & $\so4$ &  $\u2$  & $-2$ \\ \hline

 $\Lg_2$ & $\so4$ &  $\so3$  & $0$ \\ \hline

 $\Lg_2$ & $\so4$ &  $\R + \R$  & $2$ \\ \hline
%

%-----------------------------------------------------------------

\end{tabular}
\bl {\sc Table~\ref{TSymSym}}. (cont.)
\end{table}

%%%%%%%%%%%%%%%%%%%%%%%%%%%%%%%%%%%%%%%%%%%%%%%%%%%%%%%%%%%%%%%%%%%%%%%%%%%%%%%%

\begin{lemma}\label{LmNonEx}
There are no \ZZss\ of type F\,I-II-II, E\,I-III-IV, E\,II-IV-IV, E\,V-VI-VI,
E\,VI-VI-VII E\,V-VII-VII.
\end{lemma}

\begin{proof}
The non-existence of type F\,I-II-II was already shown in Example~\ref{Eg2}.
The non-existence of type E\,I-III-IV or E\,II-IV-IV spaces follows immediately
from Tables~\ref{THermann} and \ref{TSymSym}. For a pair of involutions of type
E\,VI-VI on $\LE_7$ we have $d = -5$, see Table~\ref{THermann}. It follows from
Table~\ref{TSymSym} that there is no symmetric subalgebra $\k$ of $\h = \su8$
or $\h = \Le_6 + \R$ such that $c(\su8,\k)=-5$; this shows there are no \ZZss\
of type E\,V-VI-VI or E\,VI-VI-VII. Now consider a pair of involutions of type
E\,VII-VII on $\LE_7$; by Table~\ref{THermann} we have $d = -25$ and it follows
from Table~\ref{TSymSym} that $\k \cong \so{10} + \R + \R$ or $\k \cong \Lf_4$;
since neither occurs as a subalgebra of $\su8$, we have proved the
non-existence of a \ZZs\ of type E\,V-VII-VII.
\end{proof}

\begin{remark}\label{RemConj}\rm
In each row of Table~\ref{TExcSpaces}, only the isomorphism type of the
subalgebra $\k \subset \g$ and the type of the \ZZs\ (i.e.\ the entry
in the first column defines the isomorphism type of the three symmetric
subalgebras $\gs, \gt, \gst$ of~$\g$) is given. However, this information is sufficient
to determine the conjugacy class of the subalgebra $\k \subset \g$ in most
cases, since $\k \subset \h$ is also symmetric subalgebra for $\h = \gs, \gt,
\gst$: Assume that $\g$ is a compact Lie algebra and let $\k \subset \h \subset
\g$ and $\k' \subset \h' \subset \g$ be such that all inclusions are
of symmetric subalgebras and such that $\k \cong \k'$, $\h \cong \h'$. Then
there is an automorphism $\eta \in \Aut(\g)$ such that $\eta(\h') = \h$. In
most cases (cf.\ \cite{hyperpolar}, 3.1) there is an inner automorphism
$\theta_0$ of $\h$ such that $\theta_0(\eta(\k'))=\k$. Then obviously
$\theta_0$ can be extended to an inner automorphism $\theta$ of~$\g$ such that
$\theta \circ \eta(\k') = \k$.

This shows that for most cases the subalgebra $\k$ is uniquely defined by the
data in Table~\ref{TExcSpaces} . The only exception occurs for the \ZZs\ of
type E\,VI-VI-VI with $\k \cong \u6 + \R$. Indeed, there are two conjugacy
classes of subalgebras isomorphic to $\La_5$ in $\Le_7$.  These two classes are
denoted as $[\LA_5]'$ and $[\LA_5]''$ in~\cite{dynkin1}. They are given by the
two diagrams below.
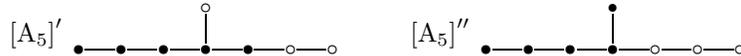
\begin{center}
$\lbk \LA_5 \rbk'$\begin{minipage}{108pt}
\begin{picture}(100,45)
\multiput(6,17)(16,0){5}{\circle*{3}} \multiput(6,17)(16,0){7}{\circle{3}}
\multiput(8.5,17)(16,0){6}{\line(8,0){11}}
\multiput(54,19.5)(0,16){1}{\line(0,8){11}} \put(54,33){\circle{3}}
\put(4,5){$$} \put(20,5){$$} \put(36,5){$$} \put(52,5){$$} \put(68,5){$$}
\put(44,30){$$}
\end{picture}
\end{minipage}\hspace{2em}
$\lbk \LA_5\rbk''$\begin{minipage}{108pt}
\begin{picture}(100,45)
\multiput(6,17)(16,0){4}{\circle*{3}} \multiput(6,17)(16,0){7}{\circle{3}}
\multiput(8.5,17)(16,0){6}{\line(8,0){11}}
\multiput(54,19.5)(0,16){1}{\line(0,8){11}} \put(54,33){\circle*{3}}
\put(4,5){$$} \put(20,5){$$} \put(36,5){$$} \put(52,5){$$} \put(68,5){$$}
\put(44,30){$$}
\end{picture}
\end{minipage}\end{center}
Note that these two subalgebras of $\Le_7$ are both contained in~$\so(12)
\subset \Le_7$ and, as subalgebras of~$\so(12)$, they are conjugate by an outer
automorphism of $\so(12)$; however, this automorphism cannot be extended to an
automorphism of~$\Le_7$. It follows from Tables~25 and~26, p.~204, 207
in~\cite{dynkin1} that only the subalgebra $\lbk \LA_5\rbk''$ gives rise to a
\ZZs.
\end{remark}

\begin{proof}[Proof of Theorem~\ref{ThMain1}]
For each entry in Table~\ref{TExcSpaces} we have given a construction in
Section~\ref{Constructions} thus showing the existence of a corresponding \ZZs.
Lemma~\ref{LmNonEx} shows that the types of \ZZs\ not listed in
Table~\ref{TExcSpaces} indeed do not exist. (Note that on $\LE_6$, the
involutions of type E\,I and E\,IV are outer, while the involutions of type
E\,II and E\,III are inner; in particular, there cannot be \ZZss\ of type
E\,I-I-I or of type E\,I-II-II etc.)

However, it remains to be shown that the type of the \ZZs\ as given in
Table~\ref{TExcSpaces} indeed determines the conjugacy class of the fixed point
set $\k = \gs \cap \gt$ except for spaces of type E\,VI-VI-VI. In most cases,
it can be seen from Tables~\ref{THermann} and \ref{TSymSym} that the conjugacy
class of the subalgebra $\k \subset \g$ is uniquely determined by the type of
the \ZZs, i.e.\ in most cases, the conjugacy class of the subalgebra $\k$ is
uniquely determined by the triple $(\g,\h, c(\h,\k))$ where $c(\h,\k) = d$.
There are three exceptions to this general rule:

For $\g = \Le_6$ we have two entries in Table~\ref{TSymSym} with $\h = \su6 +
\sp1$ and $c(\h,\k)=2$, namely $\k = \su3 + \su3 + \R + \R$ and $\k = \so6 +
\sp1$. For the first possibility we have already shown the existence of a
corresponding \ZZs\ of type~E\,II-II-II.  The second possibility would imply
the existence of a \ZZs\ of type~E\,II-II-II with $\k = \so6 + \sp1 \subset
\su6 + \sp1$ such that the isotropy representation of~$\k$ on~$\g / \k$
restricted to~$\so6$ consists of a sum of three isotypical summands; however
this is not the case according to \cite{dynkin1}, Tables~25, 26, pp.~200, 206.

There are also two entries in Table~\ref{TSymSym} with $\h = \su6 + \sp1$ and
$c(\h,\k)= -6$, namely $\k = \suxu42 + \sp1$ and $\k = \sp3 + \R$. For the
first we have already shown the existence of a \ZZs\ of type E\,II-II-III in
Example~\ref{Eg1}; the second one can be ruled out since $\sp3 + \R$ is not a
symmetric subalgebra of~$\so(10) + \R$.

Finally, for $\g = \Le_7$ there are two items in Table~\ref{TSymSym} with $\h =
\so(12) + \sp1$ and $c(\h,\k)=-5$, namely $\k = \so8 + \so4 + \sp1$ and $\k =
\u6 + \R$. We have shown the existence of a corresponding \ZZs\ of type
E\,VI-VI-VI for both cases in Examples~\ref{Eg1} and~\ref{Eg2}, respectively.
\end{proof}

%%%%%%%%%%%%%%%%%%%%%%%%%%%%%%%%%%%%%%%%%%%%%%%%%%%%%%%%%%%%%%%%%%%%%%%%%%%%%%%%

\section{The classification in the {\rm Spin(8)} case}
\label{SpinClass}

%%%%%%%%%%%%%%%%%%%%%%%%%%%%%%%%%%%%%%%%%%%%%%%%%%%%%%%%%%%%%%%%%%%%%%%%%%%%%%%%

In \cite{bg}, the \ZZs\ of the classical groups have been determined except for
$G \cong \Spin8$. This case requires special attention since the group of outer
automorphisms $\Out(\Spin8) = \Aut(\Spin8)/\Inn(\Spin8)$ is isomorphic to the
symmetric group on three letters, whereas for all other simple compact Lie
groups the group $\Out(G)$ is either trivial or isomorphic to~$\Z_2$. In this
section, we will classify the \ZZss\ with $G = \Spin8$. It was remarked
in~\cite{bg} that the result can be obtained from~\cite{dc}. We give an
independent proof, using an analogous method as in the exceptional case.

We recall~\cite{hyperpolar} that all connected symmetric subgroups
of~$\Spin(8)$ are conjugate by some automorphism of $\Spin(8)$ to one of the
groups
$$
\Spin(8-n) \cdot \Spin(n), \qquad n = 1,2,3,4.
$$
However, if one considers conjugacy classes with respect to inner
automorphisms, there are three conjugacy classes of symmetric subgroups locally
isomorphic to one of $\Spin7$, $\Spin6 \cdot \Spin2$, $\Spin5 \cdot \Spin3$,
respectively, and only one conjugacy class of subgroups locally isomorphic to
$\Spin4 \cdot \Spin4$, cf.~\cite{hyperpolar}, Proposition~3.3. Hence there are
ten conjugacy classes (w.r.t.\ inner automorphisms) of connected symmetric
subgroups in $\Spin8$; under the covering map $\Spin8 \to \SO8$, they
correspond to the conjugacy classes (w.r.t.\ inner automorphisms) of connected
(locally) symmetric subgroups in $\SO8$ given by $\SO7$, $\Spin7^+$,
$\Spin7^-$, $\SO6 \times \SO2$, $\U4$, $\a(\U4)$, $\SO5 \times \SO3$, $\Sp2
\cdot \Sp1$, $\a(\Sp2 \cdot \Sp1)$, $\SO4 \times \SO4$, where
$\a=i_{\diag(-1,1,\ldots,1)}$.

%%%%%%%%%%%%%%%%%%%%%%%%%%%%%%%%%%%%%%%%%%%%%%%%%%%%%%%%%%%%%%%%%%%%%%%%%%%%%%%%
\begin{table}[h]\rm
\begin{tabular}{|c|c|c|}
\hline \str $G^{\s}$ & $G^{\t}$ & $d$ \\
\hline \hline

 $\Spin7$ & $\Spin7$ & $-14$ \\ \hline

 $\Spin7$ & $\Spin6 \cdot \Spin2$ & $-9$ \\ \hline

 $\Spin7$ & $\Spin5 \cdot \Spin3$ & $-6$ \\ \hline

 $\Spin7$ & $\Spin4 \cdot \Spin4$ & $-5$ \\ \hline

 $\Spin6 \cdot \Spin2$ & $\Spin6 \cdot \Spin2$ & $-4$ \\ \hline

 $\Spin6 \cdot \Spin2$ & $\Spin5 \cdot \Spin3$ & $-1$ \\ \hline

 $\Spin6 \cdot \Spin2$ & $\Spin4 \cdot \Spin4$ & $0$ \\ \hline

 $\Spin5 \cdot \Spin3$ & $\Spin5 \cdot \Spin3$ & $2$ \\ \hline

 $\Spin5 \cdot \Spin3$ & $\Spin4 \cdot \Spin4$ & $3$ \\ \hline

 $\Spin4 \cdot \Spin4$ & $\Spin4 \cdot \Spin4$ & $4$ \\ \hline

\end{tabular}
\bl

\caption{Pairs of involutions on Spin(8).} \label{THSpinAcht}

\end{table}
%%%%%%%%%%%%%%%%%%%%%%%%%%%%%%%%%%%%%%%%%%%%%%%%%%%%%%%%%%%%%%%%%%%%%%%%%%%%%%%%

\begin{example}\label{EgSpin81}\rm
We may construct standard examples of \ZZss\ for $\g = \so8$ using a pair of
commuting involutions both given by conjugation with diagonal matrices of the
form $\diag \lp \pm 1,\ldots, \pm 1 \rp$. In this way obtain \ZZss\ with $\k =
\so{n_1} + \so{n_2} + \so{n_3}$, $n_1 + n_2 + n_3 = 8$, $n_i \ge 1$ or $\k =
\so{n_1} + \so{n_2} + \so{n_3} + \so{n_4}$, $n_1 + n_2 + n_3 + n_4 = 8$, $n_i
\ge 1$.
\end{example}

\begin{example}\label{EgSpin82}\rm
Now assume the subalgebra $\so6 + \so2 \cong \u4 \subset \so8$ is embedded as
$$
\lbc \left(
       \begin{array}{cc}
         A & -B^t \\
         B & A \\
       \end{array}
     \right) \mid A + iB \in \u4 \rbc.
$$
We may define a pair of commuting involutions of~$\so8$ using conjugation
with the matrices
$$
J = \left(
       \begin{array}{cc}
          & -I \\
         I &  \\
       \end{array}
     \right)\quad \mbox{and}\quad
A = \left(
       \begin{array}{cc}
         D & 0 \\
         0 & D \\
       \end{array}
     \right),
$$
respectively, where $I$ is the $4 \times 4$-identity matrix and $D$ is either
$\diag(-1,1,1,1)$ or $\diag(-1,-1,1,1)$. We obtain \ZZss\ with $\k = \u1 + \u3$
or $\k = \u2 + \u2$. Note however that the subalgebra $\u2 + \u2$ of $\so8$ is
conjugate to $\so4 + \so2 + \so2$ via an outer automorphism.
\end{example}

\begin{proof}[Proof of Theorem~\ref{ThMain2}]
Assume $\s$ and $\t$ are two commuting involutions of~$\Spin8$. Then the
$\g^{\s\t}$ is a symmetric subalgebra of $\g = \so8$ by
Proposition~\ref{PropSplitting}. Hence we may assume (by applying an
automorphism of $\so8$) that $\g^{\s\t}$ is one of the subalgebras $\h$ of
$\so8$ as given Table~\ref{TSymSpin}. Since $\gs \cap \gt \subset \h$ is a
symmetric subalgebra of $\h$, it is conjugate (by some automorphism~$\theta_0$
of $\h$) to one of the subalgebras $\k$ as from Table~\ref{TSymSpin} such that
$c(\h,\k) = d(\s,\t)$.

Now we show that any automorphism~$\theta_0$ of $\h$ can be extended to an
automorphism $\theta$ of $\so8$ such that $\theta(\h) = \h$ and
$\theta|{\h}=\theta_0$: This is obvious for inner automorphisms of~$\h$ and
there is only one case where $\h$ has outer automorphisms, namely $\h = \so4 +
\so4 \cong 4 \cdot \La_1$. After choosing a Cartan subalgebra of $\g$ contained
in $\h$, and a set of simple roots $\a_1,\,\ldots,\a_4$ we can define diagram
automorphisms which act by permuting the simple roots $\a_1, \a_2, \a_3$ and
hence by permuting the corresponding three of the four $\La_1$-summands of
$\h$. In addition, consider the inner automorphism $i_g$ with $g = \lp
\begin{smallmatrix} &I\\I&
\end{smallmatrix} \rp$, $I = \diag(1,1,1,1)$. This acts on $\h = \so4 + \so4$ by
interchanging the two $\so4$-summands. These automorphisms together with the
inner automorphisms of $\h$ extended to $\g$ generate a subgroup of $\Aut(\g)$
isomorphic to $\Aut(\h)$.

Hence there is an automorphism~$\varphi$ of~$\so8$ such that $\varphi(\gs \cap
\gt)$ is one the subalgebras $\k \subset \so8$ as given in
Table~\ref{TSymSpin}. Conversely, it is straightforward to see that for each
candidate $(\h,\k)$ one can construct a corresponding \ZZs\ using
Examples~\ref{EgSpin81} or \ref{EgSpin82}.
\end{proof}

%%%%%%%%%%%%%%%%%%%%%%%%%%%%%%%%%%%%%%%%%%%%%%%%%%%%%%%%%%%%%%%%%%%%%%%%%%%%%%%%
\begin{table}[h]\rm
\begin{tabular}{|c|c|c|c|}
\hline \str $\g$ & $\h$ & $\k$ & $c(\h,\k)$   \\
\hline \hline

%-----------------------------------------------------------------

 $\so(8)$ & $\so7$ &  $\so(7-n) + \so n$  & $-9, -1, 3$ \\
 $$ & $$ &  $n=1,2,3$  & $$ \\ \hline

 $\so(8)$ & $\so6 + \so2$ &  $\so(6-n) + \so n + \so2$  & $-6, 0, 2$ \\
 $$ & $$ &  $n=1,2,3$  & $$ \\ \hline

 $\so(8)$ & $\so6 + \so2$ &  $\so(6-n) + \so n$  & $-14, -4, 2, 4$ \\
 $$ & $$ &  $n=0,1,2,3$  & $$ \\ \hline

 $\so(8)$ & $\so6 + \so2$ &  $\u3 + \u1$  & $-4$ \\ \hline

 $\so(8)$ & $\so6 + \so2$ &  $\u3$  & $-2$ \\ \hline

 $\so(8)$ & $\so5 + \so3$ &  $\so(5-n) + \so n + \so 3$  & $-5, -1$ \\
 $$ & $$ &  $n=1,2$  & $$ \\ \hline

 $\so(8)$ & $\so5 + \so3$ &  $\so(5-n) + \so n + \so2$  & $-9, -1, 3$ \\

 $$ & $$ &  $n=0,1,2$  & $$ \\ \hline

 $\so(8)$ & $\so4 + \so4$ &  $(4-n)\cdot\sp1 + n\cdot\so2$  & $-8, -4, 0, 4$ \\
 & $$ $$ &  $n=1,2,3,4$  & $$ \\ \hline

 $\so(8)$ & $\so4 + \so4$ &  $\so3 + \so4$  & $-6$ \\ \hline

 $\so(8)$ & $\so4 + \so4$ &  $\so3 + \so3$  & $0$ \\ \hline

 $\so(8)$ & $\so4 + \so4$ &  $\so3 + \u2$  & $-2$ \\ \hline

 $\so(8)$ & $\so4 + \so4$ &  $\so3 + \so2 + \so2$  & $2$ \\ \hline

\end{tabular}
\bl \caption{Symmetric subalgebras of symmetric subalgebras of $\mathfrak s
\mathfrak o (8)$.} \label{TSymSpin}
\end{table}

%%%%%%%%%%%%%%%%%%%%%%%%%%%%%%%%%%%%%%%%%%%%%%%%%%%%%%%%%%%%%%%%%%%%%%%%

\bibliographystyle{amsplain}

\pagestyle{empty}
\end{document}